\documentclass[reqno]{amsart} 

\usepackage{general,strings,text_shortcuts}
\usepackage{amsfonts,amsthm,amsmath,amssymb}
\usepackage{verbatim,layout}
\usepackage[ansinew]{inputenc}

\usepackage[bookmarks, colorlinks=true, pdfstartview=FitV, linkcolor=blue, citecolor=blue, urlcolor=blue]{hyperref}
\usepackage{cite}  

\newcommand{\simtset}{\{\simt_1,\ldots,\simt_N\}}
\renewcommand{\ge}{\geq}
\renewcommand{\R}{\mathbb R}
\newcommand{\N}{\mathbb N}

\newcommand{\env}{\ensuremath{E}\xspace}  
\newcommand{\ghull}{\ensuremath{K}\xspace}  
\renewcommand{\diam}{\mathrm{diam}\,}
\renewcommand{\emptyset}{\varnothing}

\newcommand{\inn}[1]{\mathrm{int}\,#1}   
\newcommand{\bd}{\mathrm{bd}\,} 
\newcommand{\eps}{\ensuremath{\varepsilon}\xspace}
\newcommand{\extr}{\operatorname{extr}} 

\newcommand{\linenopar}{}  

  \let\oldmarginpar\marginpar
  \renewcommand\marginpar[1]
    {\-\oldmarginpar[\raggedleft\footnotesize #1]%
    {\raggedright\footnotesize #1}}

\numberwithin{equation}{section} \numberwithin{theorem}{section}

\allowdisplaybreaks

\begin{document}

\title{Geometry of canonical self-similar tilings}

\author{Erin P. J. Pearse}
\address{\scriptsize Department of Mathematics, 25L MacLean Hall, University of Iowa, Iowa City, IA 52246} \email{erin-pearse@uiowa.edu}

\author{Steffen Winter}
\address{\scriptsize Institut f\"{u}r Algebra und Geometrie, Universit\"{a}t Karlsruhe (TH), 76131 Karlsruhe, Germany} \email{winter@math.uni-karlsruhe.de}

\thanks{The work of EPJP was partially supported by the University of Iowa Department of Mathematics NSF VIGRE grant DMS-0602242. The work of SW was partially supported by Cornell University and a grant from the German Academic Exchange Service (DAAD)}

\date{\today}



\begin{abstract}
  We give several different geometric characterizations of the situation in which the parallel set $F_\varepsilon$ of a self-similar set $F$ can be described by the inner $\varepsilon$-parallel set $T_{-\varepsilon}$ of the associated canonical tiling $\mathcal T$, in the sense of \cite{SST}. For example, $F_\varepsilon=T_{-\varepsilon} \cup C_\varepsilon$ if and only if the boundary of the convex hull $C$ of $F$ is a subset of $F$, or if the boundary of $E$, the unbounded portion of the complement of $F$, is the boundary of a convex set.  In the characterized situation, the tiling allows one to obtain a tube formula for $F$, i.e., an expression for the volume of $F_\varepsilon$ as a function of $\varepsilon$. On the way, we clarify some geometric properties of canonical tilings.

 Motivated by the search for tube formulas, we give a generalization of the tiling construction which applies to all self-affine sets $F$ having empty interior and satisfying the open set condition. We also characterize the relation between the parallel sets of $F$ and these tilings.

\end{abstract}

\maketitle
\setcounter{tocdepth}{1}
{\small \tableofcontents}


\section{Introduction}

As the basic object of our study is a self-affine system and its attractor, the associated self-affine set, we begin by defining these terms.

\begin{defn}\label{def:self-affine-system}
  For $j=1,\dots,N$, let $\simt_j: \bRd \to \bRd$ be an affine contraction whose eigenvalues \gl all \sat $0 < \gl < 1$. Then $\{\simt_1,\dots,\simt_N\}$ is a \emph{self-affine iterated function system}.
\end{defn}

\begin{defn}\label{def:self-similar-system}
  A \emph{self-similar system} is a self-affine system for which each mapping is a similitude, i.e.,
  \linenopar
  \begin{align}\label{eqn:def:similitude}
    \simt_j(x) := r_j A_j x + a_j,
  \end{align}
  where for $j=1,\dots,N$, we have $0 < r_j < 1$, $a_j \in \bRd$, and $A_j \in O(d)$, the orthogonal group of $d$-\dimnl Euclidean space \bRd. The numbers $r_j$ are referred to as the \emph{scaling ratios} of $\simtset$.
\end{defn}

Let \attr be the self-affine set generated by the mappings $\simt_1,\dots, \simt_N$, i.e., the unique (nonempty and compact) set satisfying $\simt(F)=F$ where $\simt$ is the set mapping
\linenopar
\begin{align}\label{eqn:def:simt}
  \simt := \bigcup_{j=1}^N \simt_j.
\end{align}
The existence and uniqueness of the set \attr is ensured by the classic results of Hutchinson in \cite{Hut}. It is shown in \cite{SST} that when a self-affine system satisfies the tileset condition (TSC) and the nontriviality condition (given here in Definitions~\ref{def:tileset-condition} and \ref{def:nontriviality-condition}, respectively), then there is a natural tiling of the convex hull $\hull=[\attr]$. That is, $\simtset$ generates a decomposition of \hull into sets $\tiling=\{\tile^n:n \in\N\}$, in the sense that
\linenopar
\begin{align*}
  \cj{\bigcup\nolimits_{n=1}^\iy \tile^n} = \hull,
  \q\text{and}\q \tile^n \cap \tile^m = \es, \text{ for } n \neq m,
\end{align*}
cf.~Definition~\ref{def:opentiling}. One of our main objectives in this paper is to explore the consequences of these two conditions and characterize some properties of the tilings. In particular, we clarify the relationship between the tileset condition as defined in \cite{SST} and the open set condition, in fact, the latter is implied by the former, cf.~Proposition~\ref{prop:TSC=>OSC}. 
%
The nontriviality condition forbids self-similar sets with convex attractors, like the square or interval. Additionally, we show in Proposition~\ref{thm:nontriviality-means-convexity} that the nontriviality condition ensures the existence of tiles in the tiling construction.  
Under TSC, nontriviality is also equivalent to \attr having empty interior, see Proposition~\ref{thm:nontriv-implies-empty-interior}.  We discuss the boundary of the tiling and its Hausdorff dimension in Proposition~\ref{prop:dim-bd-T} and Remark~\ref{rem:general-tiling-boundary-dimension}.

In \cite{SST}, it was noted that the tiling \tiling constitutes the bulk of the nontrivial portion of the complement of \attr, and consequently, that one may be able to study the \eps-parallel sets (or \eps-neighborhoods) of \attr by considering the inner \eps-neighborhoods of the tiling. By the \eps-parallel set $A_\eps$ of a set $A \ci \bRd$ we mean all points not in the interior of $A$ but with distance at most \eps to $A$. (Note that our usage of $A_\eps$ differs from the usual one, where the interior points of $A$ are included, but it is more convenient for our purposes.) Similarly, the inner $\eps$-parallel set $A_{-\eps}$ consists of the points of the closure of $A$ within distance $\eps$ of $\bd A$, see Definition~\ref{def:eps-parallel-set} for details. We determine the conditions under which the tiling allows an (almost disjoint, cf. \eqref{eqn:eps-parallel-set-decomp}) decomposition of $\attr_\eps$ of the following form:
\begin{equation}\label{eqn:intro-parallel-decomp}
F_\eps = T_{-\eps}\cup C_\eps.
\end{equation}
Here $T:=\bigcup R^n$ denotes the union of the tiles of \tiling.
In Theorems~\ref{thm:parallel-set-compatibility} and~\ref{thm:relating-ext-bd-to-compatibility-thm}, we give eight equivalent conditions which characterize this state of affairs; these results will be collectively referred to as the Compatibility Theorem.

In \S\ref{sec:Generalization-of-the-tiling-construction} we generalize the tiling construction introduced in \cite{SST} and discussed in earlier sections of the present paper. Specifically, we replace the tileset condition 
with the less restrictive open set condition (see Definition~\ref{def:OSC}) and replace the convex hull with an arbitrary feasible open set.
Finally, in \S\ref{sec:Generalizing-the-compatibility-theorem} we extend the Compatibility Theorem to the generalized self-similar tilings developed in \S\ref{sec:Generalization-of-the-tiling-construction}. For instance, the tiling generated from a feasible set $O$ is compatible if and only if $\bd \cj O \ci \attr$.

Compatibility allows one to employ the tiling to obtain a tube formula for \attr and this is the driving motivation for the current paper. By a \emph{tube formula} of a set $A\subset\R^d$, we mean an expression which gives the Lebesgue volume $V(A_\eps)$ of $A_\eps$ as a function of $\eps$. Such objects are of considerable interest in spectral theory and geometry; see \cite{FractalTubeSurvey} and \cite{TFCD}, as well as the more general references \cite{We}, \cite{Gr} and \cite{Schn}. In convex geometry, tube formulas are better known as Steiner formulas:
\begin{equation}\label{eqn:Steiner-formula}
  V(A_\eps)=\sum_{k=0}^{d-1} \eps^{d-k}\kappa_{d-k} C_k(A)
\end{equation}
For compact convex subsets $A$ of $\R^d$, $V(A_\eps)$ is a polynomial in \eps and the coefficients $C_k(A)$ are called \emph{total curvatures} or \emph{intrinsic volumes}; these are important geometric invariants of the set $A$ and are related to the integrals of mean curvature provided the boundary of $A$ is sufficiently smooth.
A polynomial expansion similar to \eqref{eqn:Steiner-formula} is known for sets of \emph{positive reach} \cite{Fed}.  Also for polyconvex sets (finite unions of convex sets) und certain unions of sets with positive reach polynomial expansions are known. However, in these latter cases, the polynomial describes a ``weighted'' parallel volume which counts the points in the parallel sets with different multiplicities given by an index function; cf.~\cite{Schn, Za2}.

For more singular sets like fractals one cannot expect such polynomial behavior. Tube formulas for subsets $A$ of $\R$ have been extensively studied, see \cite{FGCD} and the references therein, and they have been related to the theory of complex dimensions. Here the tube formulas typically take the form of an infinite sum. 
In \cite{TFCD} a first attempt was made to generalize this theory to higher dimensions and tube formulas have been obtained for so called \emph{fractal sprays}. The theory is developed further in \cite{Pointwise}.
A self-similar tiling \tiling is a certain kind of fractal spray and so this theory applies. One can associate a \emph{geometric zeta function} $\gz_\tiling:\bC\times(0,\iy) \to \bC$ which encodes all the geometric information of \tiling. The \emph{complex dimensions} \sD of the tiling are the poles of $\gz_\tiling$. Then for $T=\bigcup R^n$, a tube formula (describing the inner \eps-parallel volume of the union of the tiles) of the following form holds
\[V(T_{-\eps})
  = \sum_{w\in\sD} \res[s=w]{\gz_{\tiling}(s,\eps)},
\]
see \cite{FractalTubeSurvey,TFCD,Pointwise} for details. Under mild
additional assumptions, a factor $\eps^{d-w}$ can be separated from each residue, and the formula takes a form very similar to the Steiner formula:
\begin{equation}\label{eqn:tube-formula-as-Steiner}
  V(T_{-\eps}) = \sum_{w\in\sD} \eps^{d-w} C_w(\tiling),
\end{equation}
with coefficients $C_w$ independent of \eps. Just as in \eqref{eqn:Steiner-formula}, it turns out that \sD always contains $\{0,1,2,\dots,d-1\}$. 

In \cite{SteffenThesis}, the author develops a theory of fractal curvatures: a family of geometric invariants $C_k^f(\attr)$, $k=0,1,\dots,d$. The fractal Euler characteristic $C_0$ was introduced in \cite{LlorWin}, and $C_d^f$ coincides with the Minkowski content. These curvatures are defined for certain self-affine fractals and provide a fractal analogue of the coefficients $C_k(A)$ mentioned in \eqref{eqn:Steiner-formula}. Indeed, they are even localizable as \emph{curvature measures} in the same way that the coefficients of the Steiner formula are; cf. \cite{Schn}. However, the fractal analogue of the Steiner formula is absent from the context of \cite{SteffenThesis}, and it is a major impetus for this paper to establish such a link. In particular, the methods of the present paper and the theory of fractal curvatures are both applicable when the envelope (introduced in Definition~\ref{def:envelope}) is polyconvex. It remains to be determined if the coefficients $C_w(\tiling)$ appearing in \eqref{eqn:tube-formula-as-Steiner} can thus be interpreted as curvatures, and if so, if they are compatible with the theory of \cite{SteffenThesis,LlorWin}. 

The Compatibility Theorems of the present paper describe how the parallel sets of the tilings are related to the parallel sets of \attr.
For compatible sets \attr, a tube formula for \attr is obtained from the decomposition \eqref{eqn:intro-parallel-decomp} as the sum of the (inner) tube formula of an appropriate tiling $\tiling$ and a ``trivial'' part, describing the ``outer'' parallel volume of the tiled set, i.e., the convex hull \hull of \attr:
\begin{equation}\label{eqn:intro-tube-decomp}
  V(\attr_\eps)=V(T_{-\eps}) + V(\hull_\eps).
\end{equation}
Here, $V(T_{-\eps})$ is as in \eqref{eqn:tube-formula-as-Steiner} and $V(\hull_\eps)$ is as in \eqref{eqn:Steiner-formula}. A similar formula holds for the generalized tilings when a compatible feasible set exists. The Compatibility Theorems characterize the situation in which the decomposition \eqref{eqn:intro-tube-decomp} holds; they also show the limitations of this approach. We illustrate this with suitable counterexamples (see Proposition~6.3).




\subsection*{Acknowledgements}
The authors are grateful for helpful discussions, friendly advice, and useful references from Christoph Bandt, Kenneth Falconer, John Hutchinson, Michel Lapidus, Mathias Mesing, Sze-Man Ngai, Bob Strichartz, Luke Rogers, Huo-Jun Ruan, Sasha Teplyaev, Yang Wang, and Martina Z\"{a}hle.


\section{Tileset condition and nontriviality condition}

The open set condition is a classical separation condition for the study of self-similarity, cf. \cite{Fal1}. 

\begin{defn}\label{def:OSC}
  A self-affine system $\simtset$ satisfies the \emph{open set condition} (OSC) iff there is a nonempty open set $O \ci \bRd$ such that
  \linenopar
  \begin{align}
    \simt_j(O) &\ci O, \q j=1, 2,\dots, N \label{eqn:def:OSC-containment} \\
    \simt_j(O) &\cap \simt_k(O) = \emptyset \text{ for } j \neq k.
      \label{eqn:def:OSC-disjoint}
  \end{align}
  In this case, $O$ is called a \emph{feasible open set} for \attr.
\end{defn}

We denote the \emph{convex hull} of a set $A \ci \bRd$ (that is, the smallest convex set containing $A$), by $[A]$. In particular, we denote the convex hull of the attractor \attr of a system $\simtset$ by $\hull = [\attr]$.
\begin{remark}\label{rem:convex-hull-and-affine-hull}
\attr is always assumed to be embedded in the smallest possible ambient space, i.e.\ $\R^d={\rm aff} \attr$ is the affine hull of \attr and thus \hull is of full dimension.
\end{remark}

It was a crucial observation in \cite{SST}, that the convex hull satisfies $\simt_j(\hull) \ci \hull$, which implies the nestedness of \hull under iteration, cf.~\cite[Thm 5.1, p.~3162]{SST}:
\begin{prop} \label{prop:nestedness}
  $\simt^{k+1}(\hull) \ci \simt^{k}(\hull) \ci \hull$, for $k=1,2,\ldots$.
\end{prop}
The last proposition is reminiscent of \cite[\S5.2(3)]{Hut}. We recall the conditions introduced in \cite{SST} to ensure the existence of a canonical tiling of the convex hull of \attr, namely the tileset condition and the nontriviality condition.

\begin{defn}\label{def:tileset-condition}
  A self-affine system $\simtset$ (or its attractor \attr) satisfies the \emph{tileset condition} (TSC) iff it satisfies OSC with $\inn \hull$ as a feasible open set.
\end{defn}

\begin{prop} \label{prop:TSC=>OSC}
  \attr satisfies TSC if and only if 
  \begin{equation} \label{eqn:def:tileset-condition}
    \inn \simt_j(\hull) \cap \inn \simt_k(\hull) = \emptyset
    \text{ for } j \neq k.
  \end{equation}
  \begin{proof}
  The if-part is obvious; for the only-if-part apply Proposition~\ref{prop:nestedness}.
  \end{proof}
\end{prop}


Common examples satisfying TSC (and NTC, defined just below in Definition~\ref{def:nontriviality-condition}) include the Sierpinski gasket and carpet, the Cantor set, the Koch snowflake curve, and the Menger sponge. It is obvious from the definition that TSC implies OSC. The following examples demonstrate that the converse is not true. 

\begin{exm} \label{ex1}
  Let $\attr\ci\R$ be the self-similar set generated by
  the system $\{\simt_1,\simt_2,\simt_3\}$ where the mappings
  $\simt_j:\R\to\R$ are given by $\simt_1(x)=\frac{1}{3} x$,
  $\simt_2(x)=\frac{1}{3}x+\frac{2}{3}$ and
  $\simt_3(x)=\frac{1}{9} x + \frac{1}{9}$, respectively. Let
  $O=(0,\frac{1}{3})\cup(\frac{2}{3},1)$. Clearly, $O$ is a
  feasible open set for the OSC for \attr, since the images
  $\simt_1O=(0,\frac{1}{9})\cup(\frac{2}{9},\frac{1}{3})$,
  $\simt_2O=(\frac{2}{3},\frac{7}{9})\cup(\frac{8}{9},1)$ and
  $\simt_3O=(\frac{1}{9},\frac{4}{27})\cup(\frac{5}{27},\frac{2}{9})$
  are subsets of $O$ and pairwise disjoint. Thus \attr satisfies
  the OSC. On the other hand the TSC is not satisfied. The convex
  hull of \attr is $\hull=[0,1]$ and the sets
  $\simt_1\hull=[0,\frac{1}{3}]$ and
  $\simt_3\hull=[\frac{1}{9},\frac{2}{9}]$ strongly overlap.

  Note that it is even possible for a self-affine set to satisfy
  the \emph{strong separation condition} (that is, that the images $\simt_j(\attr)$ are pairwise disjoint) but not the tileset condition.
  An example of such a set is obtained, for instance, by
  replacing the mapping $\simt_3$ in the above example with the
  mapping $\simt'_3(x)=\frac{1}{27}x+\frac{4}{27}$. The images
  $\simt_1\attr'$, $\simt_2\attr'$ and $\simt'_3\attr'$ of the
  corresponding self-similar set $\attr'$ are pairwise disjoint,
  while the images of its convex hull have intersecting interiors.
\end{exm}

\begin{figure}
  \centering
  \scalebox{0.85}{ \includegraphics{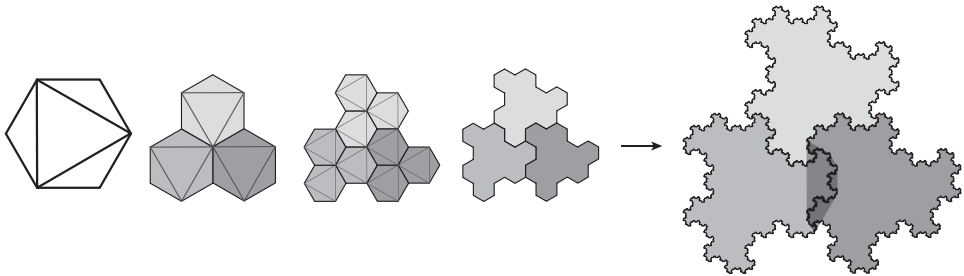}}
  \caption{\captionsize A self-similar system which satisfies the
  open set condition but not the tileset condition; see
  Example~\ref{exm:tileset_counterexample}. The attractor in
  this example tiles all of \bRt.}
  \label{fig:tileset_counterexample}
\end{figure}
%
\begin{exm} \label{ex2}
  \label{exm:tileset_counterexample}
  Consider a system of three similarity mappings, each with scaling ratio $1/\sqrt3$ and a clockwise rotation of $\gp/2$. The mappings are illustrated in Figure~\ref{fig:tileset_counterexample}. They form a system which \sats the open set condition (simply take the interior of the attractor) but not the tileset condition. On the right, the attractor has been shaded for clarity; the dark overlay indicates the intersection of the convex hulls of two first level images of the attractor.
\end{exm}

\begin{remark}\label{rem:Falconer's-question}
  After a talk on the topic of the present paper at the conference \emph{Fractal Geometry and Stochastics IV} at Greifswald, Kenneth Falconer asked the following question: Is there an easy way to decide whether for a given self-similar set \attr satisfying OSC there is a feasible open set that is convex? The results in this paper provide for the following answer. 
  
  There is a convex feasible open set for \attr if and only if \attr satisfies the tileset condition, i.e., if and only if the interior of the convex hull \hull of \attr is feasible. 
  To see this, assume that there exists a feasible open set $O$ that is convex. Then its closure $\cj{O}$ is closed and convex and satisfies $\attr \subseteq \cj{O}$ (cf.~Proposition~\ref{thm:F-in-clo(O)}). It follows that $\hull \subseteq \cj{O}$ (since the convex hull is the intersection of all closed convex sets containing \attr) and thus $\inn \hull\subseteq O$. But this implies $\simt_i(\inn \hull)\cap \simt_j(\inn \hull)\subseteq \simt_i O\cap \simt_j O=\emptyset$.
  Hence \attr satisfies the tileset condition by Proposition~\ref{prop:TSC=>OSC}.
  Thus it is sufficient to check whether the interior of the convex hull is feasible to decide the above question.
\end{remark}

\begin{defn}\label{def:nontriviality-condition}
  We say that $\simtset$ \sats the \emph{nontriviality condition} (NTC) iff its attractor $\attr$ is not convex.
\end{defn}

The nontriviality condition is, besides the TSC, the second necessary condition to ensure the existence of a canonical self-affine tiling for \attr. Proposition~\ref{thm:nontriviality-means-convexity} shows that nontriviality is precisely the condition that ensures the generators of the tiling exist, as will be apparent from Definition~\ref{def:generators}. For this reason, we say the system is \emph{trivial} iff $ \inn{(\hull)} \ci \simt(\hull)$. The following proposition shows that the present usage of ``nontrivial'' agrees with that of \cite{SST}. 


\begin{prop}\label{thm:nontriviality-means-convexity}
  A self-affine system $\simtset$ is nontrivial if and only if the images $\simt_j(\hull)$ of \hull do not cover $ \inn{(\hull)}$, i.e., the convex hull \hull \sats
  \begin{equation} \label{nontriv-cond}
 \inn{(\hull)}\nsubseteq \simt(\hull).
  \end{equation} 
\end{prop}
\begin{proof}
  First observe that (\ref{nontriv-cond}) is equivalent to
  \begin{equation} \label{nontriv-cond2}
   \hull\nsubseteq \simt(\hull).
  \end{equation}
  Indeed, the implication (\ref{nontriv-cond})$\Rightarrow$(\ref{nontriv-cond2})  is obvious.
  Conversely, if (\ref{nontriv-cond2}) holds, then $\hull\cap\simt(\hull)^\complement\neq \emptyset$. Hence there exists some point $x \in \hull\cap\simt(\hull)^\complement$ and, since $\simt(\hull)^\complement$ is open, some $\delta>0$ such that the ball $B(x,\delta)$ is contained in $\simt(\hull)^\complement$. Now, since \hull is convex and thus the closure of its interior ($\dim \hull = d$; cf. Remark~\ref{rem:convex-hull-and-affine-hull}), there is a point $y \in B(x,\delta)\cap \inn{(\hull)}$. Hence $ \inn{(\hull)}\cap\simt(\hull)^\complement$ is nonempty, implying (\ref{nontriv-cond}).

  Recall that $\simt(\hull)\ci \hull$ by nestedness (Proposition~\ref{prop:nestedness}). Therefore, if \eqref{nontriv-cond} fails, its equivalence with \eqref{nontriv-cond2} immediately implies $\hull=\simt(\hull)$. By the uniqueness of the invariant set (with respect to \simt), this means that \attr is equal to its convex hull \hull. Obviously, if the nontriviality condition is satisfied, then \attr is not equal to its convex hull. 
\end{proof}

For self-affine sets satisfying TSC, we give a different characterization of nontriviality. $\attr\subset\R^d$ is trivial if and only if it has non-empty interior.

\begin{prop}\label{thm:nontriv-implies-empty-interior}
  Let \attr be a self-affine set satisfying TSC. Then \attr is nontrivial if and only if $ \inn \attr = \es$.
\end{prop}
\begin{proof}
If \attr is nontrivial, then the set $T_0 := \inn(\hull \less \simt(\hull))$ is nonempty, but $T_0 \cap \attr=\es$, since $\attr \ci \simt(\hull)$. Observe that TSC implies $\simt_i( \inn \hull)\cap \simt_j(\attr)=\emptyset$ for $i\neq j$. Therefore, $\simt_i(T_0)\cap \attr \ci \simt_i(T_0)\cap \simt_i(\attr)=\simt_i(T_0\cap \attr)=\emptyset$ and so $\simt(T_0)\cap \attr=\emptyset$. By induction, we get $\simt^k(T_0) \cap \attr = \es$ for $k=0,1,2,\ldots$ 
Now let $x \in \attr$. Since, by the contraction principle,  $d_H(\attr,\simt^k(\cj{T_0})) = d_H(\simt^k(\attr),\simt^k(\cj{T_0})) \to 0$ as $k\to \infty$,
there exists a sequence $x_k\to x$ with $x_k \in \simt^k(\cj{T_0})=\cj{\simt^k(T_0)}$. For each $x_k$ there are points in $\simt^k(T_0)$ arbitrarily close to $x_k$. Hence $x$ cannot lie in the interior of \attr.

  For the converse, if \attr is trivial, then it is convex by Proposition~\ref{thm:nontriviality-means-convexity}. In view of Remark~\ref{rem:convex-hull-and-affine-hull}, $ \inn \attr \neq \es$.
\end{proof}

\begin{remark}
The fact that self-affine sets satisfying TSC and NTC have empty interior was used implicitly in \cite{SST} without mention. Proposition~\ref{thm:nontriv-implies-empty-interior} clarifies that this was justified.
\end{remark}

Combining Propositions~\ref{thm:nontriviality-means-convexity} and \ref{thm:nontriv-implies-empty-interior}, we infer that for self-affine sets satisfying TSC, non-empty interior means convexity. For the special case of self-similar sets, convexity is also equivalent to having full dimension. This follows from a result of Schief \cite[Cor.~2.3]{Schi} stating that for self-similar sets $\attr \ci \bRd$ satisfying OSC, $\dim_H F=d$ implies that \attr has interior points. 


\begin{cor}\label{thm:dimension-d-implies-trivial}
  Let $\attr \ci \bRd$ be a self-affine set satisfying TSC. If \attr has Hausdorff dimension strictly less than $d$, then \attr is nontrivial.
 Moreover, if \attr is self-similar, then also the converse holds.
\end{cor}
\begin{proof}
If \attr is trivial, then, by Proposition~\ref{thm:nontriv-implies-empty-interior}, \attr has non-empty interior which implies $\dim_H \attr =d$.
Now let \attr be self-similar and satisfy TSC. Assume $\dim_H \attr =d$. Since TSC implies OSC, by \cite[Theorem~2.2 and Corollary~2.3]{Schi}, \attr has non-empty interior. Therefore, by Proposition~\ref{thm:nontriv-implies-empty-interior}, \attr is trivial.
\end{proof}

%
See also Proposition~\ref{thm:nontriv-equiv-conditions} and Corollary~\ref{cor:OSC-dimension-d-implies-trivial}, for analogues of Proposition~\ref{thm:nontriv-implies-empty-interior} and Corollary~\ref{thm:dimension-d-implies-trivial} in the more general context of OSC.

 \section{Canonical self-affine tilings} 

Let $\simtset$ be a self-affine system with attractor \attr satisfying both TSC and NTC. In this section, we recall the construction of the so called \emph{canonical self-affine tiling} of the convex hull \hull of \attr introduced in \cite[\S3]{SST}.
On the way, we prove some foundational results concerning open tilings, thereby clarifying a couple of technical points which were left vague in \cite{SST}.

\begin{defn} \label{def:opentiling}
  A sequence ${\mathcal A}=\{A^i\}_{i \in\N}$ of pairwise disjoint open sets $A^i\ci\R^d$ is called an \emph{open tiling} of a set $B\ci\R^d$ (or a \emph{tiling of $B$ by open sets}) if and only if
  \[ \overline{B}= \overline{\bigcup_{i=1}^ \infty A^i}.\]
  The sets $A^i$ are called the \emph{tiles}.
\end{defn}

Note that Definition~\ref{def:opentiling} is weaker than the usual definition of a tiling: no local finiteness is assumed. In other words, a given compact set may be intersected by infinitely many of the tiles. The case that $B$ is tiled by a finite number $m \in\N$ of tiles $A^1,\dots, A^m$ is included by setting $A^i:=\emptyset$ for $i>m$. Since here we are more interested in open tilings of $B$ by an infinite number of sets, the tiles $A^i$ are usually assumed to be nonempty. Also note that each sequence $\{A^i\}$ of disjoint open sets is an open tiling of some set $B\ci\R^d$ but this set is not uniquely determined. For instance, if $\{A^i\}$ is an open tiling of $B$ then it is also an open tiling of $ \inn B$ and of $\cj{B}$. The sequence $\{A^i\}$ only determines the closure of $B$ uniquely.

The following observation regarding the boundaries of the tiles will be useful  in the sequel. In particular, it is used repeatedly in the proof of Theorem~\ref{thm:parallel-set-compatibility}, a central result of this paper. Let $\{A^i\}$ be an open tiling of a set $B$. Denote by  $A=\bigcup_{i=1}^ \infty A^i$ the union of the tiles. Since the sets $A^i$ are open, $A$ is open as well. The boundary of $A$ (defined in the usual way as $\bd A= \cj{A}\cap\cj{A^\complement}$ or, since $A$ is open, equivalently by $\bd A=\cj{A}\setminus A$) is characterized by the tiles as follows:
\begin{lemma} \label{thm:opentilinglem1}
  $\displaystyle \bd A = \cj{ \bigcup_i \bd A^i}$.
\end{lemma}
\begin{proof}
  ($\ci$): Let $x \in \bd A$. Then there
  exists a sequence $\{x_k\}_{k=1}^ \infty$ in $A=\bigcup_i
  A^i$ converging to $x$ as $k\to \infty$. Using $\{x_k\}$, we
  construct a sequence $\{x_k'\}$ in $\bigcup_i \bd A^i$ in the
  following way. For each $x_k$ there is a (unique) index
  $n(k) \in\N$ such that $x_k \in A^{n(k)}$. Since $A^{n(k)}$ is
  open, $x\notin A^{n(k)}$. Now let $x_k'$ be any point of the
  set $[x,x_k]\cap \bd A^{n(k)}$, where $[x,x_k]$ is the (closed)
  line segment between $x$ and $x_k$. Such a point exists, since
  $x \in (A^{n(k)})^\complement$ (but it may not be unique). Then, clearly,
  $\{x_k'\}$ is a sequence in $\bigcup_i \bd A^i$. Moreover,
  $x_k'\to x$ as $k\to \infty$, since $x_k\to x$ and
  $|x-x_k'|<|x-x_k|$. But this implies that
  $x \in\cj{\bigcup_i \bd A^i}$, proving the inclusion from
  left to right.

  ($\supseteq$): For a proof of the reversed inclusion, let
  $x \in\cj{\bigcup_i \bd A^i}$. Then there exists a
  sequence $\{y_k\}\ci \bigcup_i \bd A^i$ such that $y_k\to
  x$ as $k\to \infty$. The existence of this sequence (and the
  disjointness of the tiles $A^i$) imply immediately that
  $x\notin A$, since an interior point of a set can not be
  accumulation point of a sequence in its complement.
   Furthermore, each $y_k$ is an element of at least one of the sets
   $\bd A^i$. Let $n(k)$ be an index such that $y_k \in \bd A^{n(k)}$.
   For each $y_k$, we find points in $A^{n(k)}$ arbitrarily close to
   $y_k$. Choose $y_k' \in A^{n(k)}$ such that $|y_k-y_k'|<\frac{1}{k}$.
   Then $|x-y_k'|\le|x-y_k|+|y_k-y_k'|<|x-y_k|+\frac{1}{k}\to 0$ as
   $k\to \infty$.
  Thus $y_k'\to x$. Recalling that $y_k' \in A^{n(k)}$ and thus
  $\{y_k'\}\ci \bigcup_i A^i=A$, we conclude that
  $x \in \cj{A}$. Together with $x\notin A$ this
  yields $x \in \cj{A}\setminus A=\bd A$,
  completing the proof.
\end{proof}

Let
\begin{equation}\label{eqn:def:words}
  W := \bigcup_{k=0}^ \infty\{1,\dots,N\}^k
\end{equation}
denote the set of all finite words formed by the alphabet $\{1,\dots,N\}$. For any word $w=w_1 w_2\dots w_n \in W$, let $\simt_w = \simt_{w_1} \circ \simt_{w_2} \circ \dots \circ \simt_{w_n}$. In particular, if $w \in W$ is the \emph{empty word} then $\simt_w=\mathrm{Id}$.

Denote by $G_1, G_2, \ldots$ the connected components of the open set
  $T_0:= \inn{(\hull\setminus \simt(\hull))}$; $T_0=\bigcup_{q \in Q} G_q$. The index set $Q \ci \N$ may be infinite, but, since $T_0$ is open, the number of its connected components is certainly at most countable.
  
\begin{defn}\label{def:self-affine-tiling} \label{def:generators}
  The \emph{canonical self-affine tiling} associated with $\simtset$ (or with \attr) is
  \linenopar
  \begin{equation}\label{eqn:def:self-affine-tiling}
    \tiling = \{ \simt_w(G_q): w \in W, q \in  Q\}.
  \end{equation}
  The open subsets $G_q$ of \hull are called the \emph{generators} of \tiling.  It is shown in \cite[Thm.~5.16, p.~3167]{SST} that \tiling is an open tiling of $\hull = [\attr]$ in the sense of Definition~\ref{def:opentiling}, i.e.\ the sets $\simt_w(G_q)$ are pairwise disjoint and
  \[
  \hull=\cj{\bigcup_{R \in\tiling} R}.
  \]
 \end{defn}
Write $T=\bigcup_{R \in\tiling} R$ for the union of the tiles of \tiling and $\bd T$ for the boundary of this set. Clearly, $T$ is open, $\cj{T}=\hull$ and so $\bd T = \cj{T} \less T=\hull \less T$. By Lemma~\ref{thm:opentilinglem1}, we have
\begin{equation}\label{eqn:boundaryT}
\bd T = \cj{\bigcup_{R \in \tiling} \bd R}.
\end{equation}

Note that the closure in the representation \eqref{eqn:boundaryT} 
cannot be omitted. One has $\attr \ci \bd T$ (cf.~Lemma~\ref{thm:alt-generated-tiling-boundary-decomp}), while $F\not \ci \bigcup_{R \in\tiling} \bd R$. If the Hausdorff dimension $\dim_H \attr$ is strictly greater than $d-1$, then taking the closure leads to a jump of dimension. More precisely, one has the equality $\dim_H \bd T=\max\{\dim_H F, d-1\}$, as is shown in Proposition~\ref{prop:dim-bd-T}. For the proof, it is convenient to work with a slight variation of the tiling described above: It is possible to consider the set $T_0$ as the generator of a tiling, instead of its connected components $G_q$. This point of view leads to a different tiling $\tiling' := \{ \simt_w(T_0) \suth w \in W\}$ of \hull whose tiles are not necessarily connected.
  It is easily seen that $\tiling'$ is also an open tiling of \hull in the sense of Definition~\ref{def:opentiling}. Moreover, for each tile $\simt_w (T_0) \in \tiling'$, $\{\simt_w (G_q) \suth q \in Q\}$ is an open tiling of $\simt_w (T_0)$.  If $T' = \bigcup_{R \in \tiling'} R$ is the union of the tiles, then by two applications of Lemma~\ref{thm:opentilinglem1}, the boundaries of both tilings coincide:
  \[\bd T'=\cj{\bigcup_{R' \in\tiling'}\bd R'}=\cj{\bigcup_{w \in W} \bd \simt_w T_0}=\cj{\bigcup_{w \in W} \cj{\bigcup_{q \in Q} \bd \simt_w G_q}}=\bd T.\]

\begin{lemma}\label{thm:alt-generated-tiling-boundary-decomp}
  $\displaystyle \bd T= \attr \cup \bigcup\nolimits_{R \in\tiling'} \bd R.$
\end{lemma}
\begin{proof}
  $(\ce)$ From Lemma~\ref{thm:opentilinglem1}, we have $\bigcup_{R \in \tiling'} \bd R \ci \bd T'=\bd T$. For the inclusion $\attr\subseteq \bd T$, note that \tiling is an open tiling of \hull and thus
  $\attr\ci \hull= \overline{\bigcup_{R\in\tiling} R}=\overline{T}$.
  But, by \cite[Thm.~5.16, p.~3167]{SST}, $R\cap
  \attr=\emptyset$ for all $R\in\tiling$, i.e. $\attr\cap T=\emptyset$.
  Thus $\attr\ci \overline{T}\setminus T=\bd T$. 

  $(\ci)$ Let $x \in\bd T=\bd T'= \bd\left(\bigcup_{R\in\tiling'}R\right)$. There exists a sequence $(x_i)$ of points converging to $x$ such that each $x_i$ is in
  some tile $R^i \in \tiling'$. For each of these tiles $R^i$ there is a word $w{(i)} \in W$
  such that $R^i = \simt_{w{(i)}} (T_0)$. 
  Observe that
  \linenopar
  \begin{align*}
    d_H\left(R^i, \vstr[2]\simt_{w{(i)}} (\attr)\right)
    &= d_H\left(\simt_{w{(i)}} (T_0), \vstr[2]\simt_{w{(i)}} (\attr) \right)
    = r_{w{(i)}} d_H(T_0, \attr)
    \le r_{w{(i)}} \diam \hull,
  \end{align*}
  since both \attr and $T_0$ are subsets of \hull.
  For the sequence of tiles $R^i$, there are two possibilities:
  \begin{enumerate}[(i)]
    \item There is a subsequence $(i_k)$ such that $\diam (\simt_{w{(i_k)}} (T_0)) \limas{k} 0$.
    \item There is a constant $c>0$ such that $\diam (\simt_{w{(i)}} (T_0)) \ge c$ for each $i \in\N$.
  \end{enumerate}
  Case (i) implies $r_{w{(i_k)}} \to 0$, and hence $d(x_{i_k},\attr) \to 0$, so that $x \in F$.
  Case (ii) is when $x \in \bd R$ for some $R\in\tiling'$. To see this, observe that $\diam (\simt_{w} (T_0)) \ge c$ for only finitely many words $w \in W$. Hence at least one of these words occurs infinitely often in the sequence $(w{(i)})$, i.e., there is a $w \in W$ and a subsequence $(i_k)$ such that $w(i_k) = w$ for all $k$. But this implies $x_{i_k} \in \simt_w (T_0) =: R$ for all $k$ and thus $x \in \cj{R}$, since $x_{i_k} \to x$. It follows that $x \in \bd R$, since $R$ is open and $x\in\bd T'$.
\end{proof}

\begin{prop}\label{prop:dim-bd-T}
  $\dim_H \bd T = \max\{\dim_H \attr, d-1\}.$
\end{prop}
\begin{proof}
  For $T_0 = \inn{(\hull\setminus \simt(\hull))}$, 
  observe that $\bd T_0$ is a subset of $\bd \hull \cup \bigcup_j \bd \simt_j \hull$. Since \hull and $\{\simt_j \hull\}_{j=1}^N$ are convex, their boundary has dimension $d-1$. It follows that $\dim_H \bd T_0 \le d-1$ by stability and monotonicity of $\dim_H$. For the reverse inequality, note that $\bd T_0$ is the boundary of an open set in \bRd. Hence $\dim_H \bd T_0 = d-1$ and so $\dim_H \bd R=d-1$ for each $R\in\tiling'$. Now the assertion follows from Lemma~\ref{thm:alt-generated-tiling-boundary-decomp} by countable stability of the Hausdorff dimension.
  %
\end{proof}

 \section{Compatibility of the \eps-parallel sets $\attr_\eps$ and $T_{-\eps}$}
\label{sec:compatibility}

In this section, we clarify the relation between the \emph{(outer) parallel sets} of \attr and the \emph{inner parallel sets} of the associated tiling \tiling.
We characterize the situation in which these parallel sets essentially coincide, for this allows to use the tiling and the theory of complex dimensions developed in \cite{TFCD} to obtain a tube formula for \attr.

\begin{figure}
  \centering
 \includegraphics{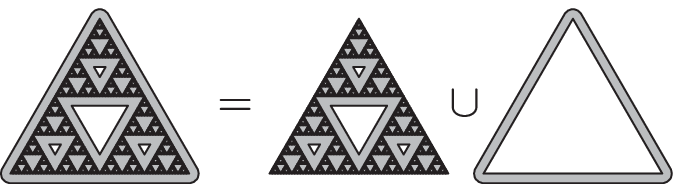}
  \caption{\captionsize The exterior \eps-\nbd of the Sierpinski gasket \attr is the union of the inner \eps-\nbd of the Sierpinski gasket tiling and the exterior \eps-\nbd of $\hull=[\attr]$. This union is disjoint except for the boundary of \hull.}
  \label{fig:gasket-compatibility}
\end{figure}

\begin{defn}\label{def:eps-parallel-set}
  For any nonempty, bounded set $A \ci \bRd$, and $\eps\ge 0$, define the (outer)
  \emph{\eps-parallel set} (or \emph{\eps-neighbourhood}) of $A$ by
  \linenopar
  \begin{align}\label{eqn:def:outer-eps-parallel-set}
    A_{\eps}:=\{x \in \cj{A^\complement} \suth d(x,A)\le\eps\}.
  \end{align}
  Similarly, define the \emph{inner \eps-parallel set} (or
  \emph{inner \eps-neighbourhood}) of $A$ by
  \linenopar
  \begin{align}\label{eqn:def:inner-eps-parallel-set}
    A_{-\eps}:=\{x \in \cj{A} \suth d(x,A^\complement)\le\eps\},
  \end{align}
  or equivalently by {\bf $A_{-\eps}= (A^\complement)_\eps$}.
\end{defn}

  Note that we do not include interior points of $A$ into the outer parallel sets, as is often done.
  For each $\eps\ge 0$, both, $A_\eps$ and $A_{-\eps}$, are always closed, bounded
  and nonempty subsets of \bRd. 
  Moreover, $A_0 = A_{-0} = \bd A \ci A_\gd$ for any $\gd \in \bR$,
  and $A_{-\eps}=A$ for $\eps\ge\rho$, where $\rho=\rho(A)$
  denotes the inradius of $A$. In particular, if
  $ \inn{A}=\emptyset$, then $A_{-\eps}=A$ for all $\eps\ge 0$.


For an open tiling ${\mathcal A}=\{A^i\}$ (cf.~Def.~\ref{def:opentiling}), denote by $A_{-\eps}$ the inner \eps-parallel set of $A:=\bigcup_i A^i$. 
\begin{lemma}\label{thm:decomposition-of-inner-epsnbd-via-tiling}
  $A_{-\eps}=\cj{\bigcup_{i=1}^ \infty A^i_{-\eps}}$.
\end{lemma}
\begin{proof}
  Let $x \in A_{-\eps}$. Then $x \in \bd A$ or there is some $l \in\N$ such that $x \in A^l$.
  In the first case, by Lemma~\ref{thm:opentilinglem1}, $x \in\bd A=\overline{ \bigcup_i \bd A^i}\ci \overline{ \bigcup_i A^i_{-\eps}}$, since $\bd A^i\subset A^i_{-\eps}$. In the latter case $d(x, (A^l)^\complement)=d(x, (\bigcup_i A^i)^\complement)\le\eps$ and thus $x \in A^l_{-\eps}\ci\cj{\bigcup_{i=1}^ \infty A^i_{-\eps}}$. Hence $A_{-\eps}\ci \cj{\bigcup_{i=1}^ \infty A^i_{-\eps}}$.

  For the reverse inclusion, let $x \in \cj{\bigcup_{i=1}^ \infty A^i_{-\eps}}$. Then there exists a sequence $y_j \in \bigcup_{i=1}^ \infty A^i_{-\eps}$ with $y_j\to x$ as $j\to \infty$. For each $j$ there is an index $i(j) \in\N$ such that $y_j \in A^{i(j)}_{-\eps}$, i.e., $y_j \in\cj{A^{i(j)}}$ and $d(x,(A^{i(j)})^\complement)\le\eps$.
  Since $A^{i(j)}\ci A$, we infer $y_j \in \cj{A}$ and $d(x, A^\complement)\le d(x, (A^{i(j)})^\complement)\le\eps$, i.e., $y_j \in A_{-\eps}$. But this implies $x \in A_{-\eps}$, since $A_{-\eps}$ is closed.
\end{proof}

Now let $\simtset$ be a self-affine system satisfying TSC and NTC, \attr its attractor and $\tiling=\{R^i\}_{i \in\N}$ the associated canonical tiling, as introduced in the previous sections. Write $T:=\bigcup_i R^i$ for the union of the tiles of \tiling. For $\eps\ge 0$, the set $T_{-\eps}$ will be regarded as the \emph{inner $\eps$-parallel set} of the tiling. 

\begin{prop} \label{parsetprop}
  Let \attr be the self-affine set associated to the system $\simtset$ satisfying TSC and NTC,
  and let \tiling be the associated canonical self-affine tiling
  of its convex hull \hull. Then
  \begin{enumerate}
    \item $\attr\ci \bd T$.
    \item $\attr_\eps \cap \hull\ci T_{-\eps}$ for $\eps\ge 0$.
    \item $\attr_\eps \cap \hull^\complement\ci \hull_\eps$ for $\eps\ge 0$.
  \end{enumerate}
\end{prop}
\begin{proof}
  (i) This is a corollary of Lemma~\ref{thm:alt-generated-tiling-boundary-decomp}. 

  (ii) Fix $\eps\ge 0$. Let $x \in \attr_\eps\cap \hull$. Then,
  since $x \in \hull=\overline{T}$, either $x \in \bd T$ or $x \in
  T$. In the former case $x \in T_{-\eps}$ is obvious, since $\bd
  T \ci T_{-\eps}$. In the latter case there exists a point
  $y \in \attr$ with $d(x,y)\le\eps$. By (i), $y$ is in $\bd T$
  and so $d(x,\bd T)\le \eps$. Hence $x \in T_{-\eps}$, completing
  the proof of (ii).

  (iii) is an immediate consequence of the inclusion
  $\attr\ci \hull$.
\end{proof}

In Theorem~\ref{thm:parallel-set-compatibility} and Theorem~\ref{thm:relating-ext-bd-to-compatibility-thm}, we characterize the situation in which one has the helpful disjoint decomposition
\linenopar
\begin{align}\label{eqn:eps-parallel-set-decomp}
  \attr_\eps = T_{-\eps}\cup (\hull_\eps\setminus \hull).
\end{align}
The decomposition \eqref{eqn:eps-parallel-set-decomp} is ensured by \eqref{itm:parsetcomp-FcapC-is-innernbd} and \eqref{itm:parsetcomp-FcapCc-is-outernbd}, and the other conditions \eqref{itm:parsetcomp-bdT-is-F}--\eqref{itm:parsetcomp-bdG-in-F} provide easy-to-check criteria for when this holds. See also Theorem~\ref{thm:relating-ext-bd-to-compatibility-thm} for two more equivalent conditions.

\begin{theorem}[Compatibility Theorem] \label{thm:parallel-set-compatibility}
  Let \attr be the self-affine set associated to the system $\simtset$ which \sats TSC and NTC. Then the following assertions are equivalent:
  \begin{enumerate}
    \item \label{itm:parsetcomp-bdT-is-F} $\bd T = \attr$.
    \item \label{itm:parsetcomp-bdC-in-F} $\bd \hull \ci \attr$.
    \item \label{itm:parsetcomp-bdPhiCcomp-in-F} $\bd (\hull\setminus \simt(\hull))\ci \attr$.
    \item \label{itm:parsetcomp-bdG-in-F} $\bd G_q \ci \attr$ for all $q \in Q$.
    \item \label{itm:parsetcomp-FcapC-is-innernbd} $\attr_\eps\cap \hull=T_{-\eps}$ for all $\eps\ge 0$.
    \item \label{itm:parsetcomp-FcapCc-is-outernbd} $\attr_\eps\cap \hull^\complement= \hull_\eps\setminus \bd \hull$ for all $\eps\ge 0$.
  \end{enumerate}
\end{theorem}
\begin{proof}
  We show the inclusions $\eqref{itm:parsetcomp-bdT-is-F}\Rightarrow \eqref{itm:parsetcomp-bdC-in-F}\Rightarrow \eqref{itm:parsetcomp-bdPhiCcomp-in-F}\Rightarrow
  \eqref{itm:parsetcomp-bdG-in-F}\Rightarrow \eqref{itm:parsetcomp-bdT-is-F}$, then $\eqref{itm:parsetcomp-bdT-is-F}\Rightarrow \eqref{itm:parsetcomp-FcapC-is-innernbd}\Rightarrow
  \eqref{itm:parsetcomp-bdG-in-F}$, and  $\eqref{itm:parsetcomp-bdC-in-F}\Leftrightarrow\eqref{itm:parsetcomp-FcapCc-is-outernbd}$.

  $\eqref{itm:parsetcomp-bdT-is-F}\Rightarrow \eqref{itm:parsetcomp-bdC-in-F}$. Observe that $\bd \hull\ci \bd T$.

  $\eqref{itm:parsetcomp-bdC-in-F}\Rightarrow\eqref{itm:parsetcomp-bdPhiCcomp-in-F}$. Assume that $\bd \hull\ci \attr$.
  Then also $\bd \simt(\hull)\ci \simt(\bd \hull)\ci
  \simt(\attr)=\attr$ and so $\bd(\hull\setminus
  \simt(\hull))\ci \bd \hull \cup \bd \simt(\hull)
 \ci \attr$. (Here we used that, for $A,B\ci\R^d$,
  $\bd(A\cup B)\ci \bd A \cup \bd B$ and $\bd(A\setminus
  B)\ci \bd A\cup \bd B$.)

  $\eqref{itm:parsetcomp-bdPhiCcomp-in-F} \Rightarrow \eqref{itm:parsetcomp-bdG-in-F}$.
  Assume that $\bd(\hull\setminus
  \simt(\hull))\ci \attr$. The generators $G_q$ (being the
  connected components of the open set $ \inn{(\hull\setminus
  \simt(\hull))}$) form an open tiling of the set
  $ \inn{(\hull\setminus \simt(\hull))}$. Therefore, by
  Lemma~\ref{thm:opentilinglem1}, $\bd G_q\ci \bigcup_q \bd G_q
 \ci \overline{\bigcup_q \bd G_q} = \bd(\bigcup_q G_q) =
  \bd( \inn{(\hull\setminus \simt(\hull))}) \ci \bd(\hull\setminus \simt(\hull))$.
  Hence $\bd G_q\ci \bd(\hull\setminus \simt(\hull))\ci
  \attr$ for each $q$, showing \eqref{itm:parsetcomp-bdG-in-F}.

  $\eqref{itm:parsetcomp-bdG-in-F}\Rightarrow \eqref{itm:parsetcomp-bdT-is-F}$. Let $\bd G_q \ci \attr$ for all
  $q \in\{1,\dots,Q\}$. It suffices to show that this implies
  $\bd T\ci \attr$, the reversed inclusion being always
  true, cf.\ Proposition~\ref{parsetprop} \eqref{itm:parsetcomp-bdT-is-F}. By definition of
  the tiles, $R^i=\simt_w G_q$ for some $w \in W$ and some $q$ and
  thus we have $\bd R^i=\bd \simt_w G_q = \simt_w \bd
  G_q\ci \simt_w \attr\ci \attr$ for each $i \in \N$.
  But this implies $\overline{\bigcup_i \bd R^i}\ci \attr$,
  since \attr is closed. Finally, since, by
  Lemma~\ref{thm:opentilinglem1}, $\bd T= \overline{\bigcup_i \bd
  R^i}$, assertion \eqref{itm:parsetcomp-bdT-is-F} follows.

  $\eqref{itm:parsetcomp-bdT-is-F}\Rightarrow \eqref{itm:parsetcomp-FcapC-is-innernbd}$. By Proposition~\ref{parsetprop} \eqref{itm:parsetcomp-bdG-in-F}, it
  suffices to show the inclusion $T_{-\eps}\ci
  \attr_\eps\cap \hull$ for each $\eps\ge0$. So fix $\eps\ge 0$
  and let $x \in T_{-\eps}$. Then, clearly, $x \in \hull$.
  Moreover, either $x \in\bd T$ or $x \in R^i$ for some $i \in\N$
  and $d(x,\bd R^i)\le\eps$. Both cases imply $x \in \attr_\eps$,
  the former since, by \eqref{itm:parsetcomp-bdT-is-F}, $\bd T=\attr\ci \attr_\eps$,
  and the latter since $\bd R^i\ci \bigcup_j \bd R^j
 \ci \bd T= \attr$ and so $d(x,\attr)\le d(x,\bd R^i)\le
  \eps$.

  $\eqref{itm:parsetcomp-FcapC-is-innernbd}\Rightarrow \eqref{itm:parsetcomp-bdG-in-F}$ (by contraposition). Assume that \eqref{itm:parsetcomp-bdG-in-F} is
  false, i.e.\ assume there exists some index $q$ and some $x \in
  \bd G_q$ such that $x\notin \attr$. Then, since \attr is
  closed, there is some number $\delta>0$ such that $d(x,
  \attr)>\delta$ and so $x\notin \attr_\eps$ for $\eps\le\delta$.
  On the other hand, $x \in \bd G_q$ clearly implies $x \in
  T_{-\eps}$. Hence the equality in \eqref{itm:parsetcomp-FcapC-is-innernbd} does not hold.

  $\eqref{itm:parsetcomp-bdC-in-F}\Rightarrow \eqref{itm:parsetcomp-FcapCc-is-outernbd}$. For $\eps=0$ there is nothing to prove.
  So let $\eps>0$ and $x \in \hull_\eps\setminus \bd \hull$. Then
  there exists a point $y \in\bd \hull$ such that  $d(x,y)\le
  \eps$. By \eqref{itm:parsetcomp-bdC-in-F}, $y \in \attr$ and thus $d(x,\attr)\le \eps$,
  i.e.\ $x \in \attr_\eps$. Hence $\hull_\eps \ci \attr_\eps
  \cap \hull^\complement$. The reversed inclusion is always true, cf.\
  Proposition~\ref{parsetprop} \eqref{itm:parsetcomp-bdPhiCcomp-in-F}, and so assertion \eqref{itm:parsetcomp-FcapCc-is-outernbd} follows.


  $\eqref{itm:parsetcomp-FcapCc-is-outernbd}\Rightarrow \eqref{itm:parsetcomp-bdC-in-F}$ (by contraposition). Assume \eqref{itm:parsetcomp-bdC-in-F} is false,
  i.e.\ there exists a point $x \in \bd \hull$ such that $x\notin
  \attr$.
  Let $\delta:=d(x,\attr)$ and fix some $\eps<\frac\delta 2$. Since $x \in\bd \hull$, there are points in $\hull^\complement$ arbitrarily close to $x$. Choose $y \in \hull_\eps\cap \hull^\complement$. Then $d(y,\attr)\ge d(x,\attr)-d(x,y)> \eps$, implying $y\not \in \attr_\eps$. Hence the equality $\attr_\eps\cap \hull^\complement=\hull_\eps\setminus \bd \hull$ can not be true for this $\eps$, i.e., \eqref{itm:parsetcomp-FcapCc-is-outernbd} does not hold.
\end{proof}
  %

Note that the assertions \eqref{itm:parsetcomp-bdC-in-F}, \eqref{itm:parsetcomp-bdPhiCcomp-in-F} and \eqref{itm:parsetcomp-bdG-in-F} are very simple and easy to check. So, in particular, Theorem~\ref{thm:parallel-set-compatibility}
states that if one of the assertions \eqref{itm:parsetcomp-bdC-in-F}, \eqref{itm:parsetcomp-bdPhiCcomp-in-F} or  \eqref{itm:parsetcomp-bdG-in-F} is true for a given self-affine set \attr, then each of its parallel sets $\attr_\eps$ is the disjoint union of the two sets $T_{-\eps}$ and $\hull_\eps\setminus \hull$, cf.~(\ref{eqn:eps-parallel-set-decomp}).
Moreover, if for some \attr, it can be shown that one of the assertions \eqref{itm:parsetcomp-bdC-in-F}, \eqref{itm:parsetcomp-bdPhiCcomp-in-F} or \eqref{itm:parsetcomp-bdG-in-F} is false, then the inner parallel set $T_{-\eps}$ of the tiling does not describe the set $\attr_\eps \cap \hull$ and also the sets $\attr_\eps\cap \hull^\complement$ and $\hull_\eps\setminus\hull$ are different. Thus, for any set \attr not satisfying the assertions of
Theorem~\ref{thm:parallel-set-compatibility}, $\attr_\eps$ does not coincide with $T_{-\eps} \cup \hull_\eps$, and one cannot use the tiling directly to study the parallel sets $\attr_\eps$.

\subsection*{The envelope}
We consider another hull operation, the \emph{envelope}, and show that the conditions of the compatibility theorem are met for a self-affine set \attr precisely when its envelope coincides with its convex hull. At the end of \S\ref{sec:Generalizing-the-compatibility-theorem}, we examine the feasibility of the envelope as a replacement for the convex hull in the tiling construction; cf. Proposition~\ref{prop:no-compatible-open-set} and the ensuing discussion. There are many cases where Theorem~\ref{thm:parallel-set-compatibility} does not apply for the tiling as constructed using the convex hull, but the analogous result (Theorem~\ref{thm:Generalized-parallel-set-compatibility}) \emph{does} apply when the convex hull is replaced by the envelope.

\begin{defn}\label{def:exterior-boundary} \label{def:envelope}
  Let $K\subset\R^d$ be a compact set. $K^\complement$ has a unique unbounded component, which we call $U$. (For $d=1$, there are actually two unbounded components in $K^\complement$, if $+ \infty$ and $- \infty$ are not identified. In this case let $U$ be their union.) Then $\bd U$ is the \emph{exterior boundary} of $K$; it consists of that portion of (the boundary of) $K$ which is accessible when approaching $K$ from infinity.
  The \emph{envelope} $\env=\env(K)$ of $K$ is the complement of $U$, $\env := U^\complement$.
\end{defn}

\begin{exm}\label{exm:envelopes}
  The envelope of the Sierpinski gasket is its convex hull, as is the envelope of the Sierpinski carpet. The envelope of the Koch curve is the Koch curve itself, as is the envelope of the attractor depicted in Figure~\ref{fig:tileset_counterexample}. Some more interesting (and non-convex) envelopes are shown in Figure~\ref{fig:envelope-examples}; for a description of these sets cf.~\cite{LlorWin}. 

  \begin{figure}
    \centering
 \includegraphics{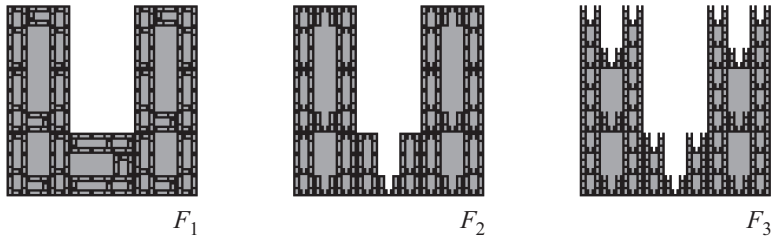}
    \caption{\captionsize Three self-similar sets and their envelopes (the shaded region, including the attractor itself). $\attr_3$ is the attractor of a system $\simt^{(1)}$ of 7 mappings, each with scaling ratio $\frac13$ and no rotation. To make $\attr_2$, we have given two of the mappings a rotation of $\gp$ (top left and top right). To make $\attr_1$, we have additionally given one of the mappings a rotation of $\frac\gp2$ (bottom center).}
    \label{fig:envelope-examples}
  \end{figure}
\end{exm}

\begin{lemma}\label{thm:envelope-wraps-attr}\label{thm:env-contained-in-hull}
Let $K\subset \R^d$ be a compact set.
 The envelope $\env$ of $K$ is compact and satisfies $\bd \env \ci K \ci \env$. Moreover, $\env \ci [K]$, where $[K]$ is the convex hull of $K$.
\end{lemma}

The following results indicate that the conditions of Theorem~\ref{thm:parallel-set-compatibility} are satisfied precisely when a self-affine set \attr appears convex when seen ``from outside''.

\begin{prop}\label{thm:exterior-boundary}
  Let $\ghull$ be a compact set in $\R^d$ with envelope $\env$ and convex hull $[\ghull]$.
  Then $\env$ is convex iff $\env =[\ghull]$.
\end{prop}
  \begin{proof}
    If $\env = [\ghull]$ then $\env$ is obviously convex.
    For the other implication, assume $\env$ is convex. By Lemma~\ref{thm:env-contained-in-hull}, we have $\env \ci [K]$ and, moreover, $K \ci \env$. The latter implies $[K]\subset [\env]$ and, since $\env=[\env]$, also the reversed inclusion $[K]\subseteq \env$ is proved.  
\end{proof}

Now we have two more ``compatibility conditions'' to accompany those already established in Theorem~\ref{thm:parallel-set-compatibility}. Let $E$ denote the envelope of \attr and let \hull be its convex hull, as before.

\begin{theorem}\label{thm:relating-ext-bd-to-compatibility-thm}
  Each of the following two conditions is equivalent to any of the conditions (i)--(vi) of Theorem~\ref{thm:parallel-set-compatibility}:
  \begin{enumerate}[{\rm(a)}]
    \item $ \env = \hull$.
    \item $ \env$ is convex.
  \end{enumerate}
  \begin{proof}
    We prove \eqref{itm:parsetcomp-bdC-in-F}$\Rightarrow$(b) and (a)$\Rightarrow$\eqref{itm:parsetcomp-bdC-in-F}, where \eqref{itm:parsetcomp-bdG-in-F} and \eqref{itm:parsetcomp-bdC-in-F} are the conditions from Theorem~\ref{thm:parallel-set-compatibility}. Note that  (a) is equivalent to (b) by Proposition~\ref{thm:exterior-boundary}.

    \eqref{itm:parsetcomp-bdC-in-F}$\Rightarrow$(b), by contraposition. If $ \env$ is not convex, then $ \inn \hull \less \env$ is nonempty and so there must exist a point $x \in \inn\,\hull\cap U$, i.e., $x$ is in the unbounded connected component $U$ of $\attr^\complement$. Hence there must be a path in $U$ connecting $x$ to infinity. Since $x \in \inn\,\hull$ this path crosses $\bd \hull$, implying the existence of a point $y \in\bd \hull$ which is not in \attr. Hence condition \eqref{itm:parsetcomp-bdC-in-F} of  Theorem~\ref{thm:parallel-set-compatibility} does not hold.

    (a)$\Rightarrow$\eqref{itm:parsetcomp-bdC-in-F}. Note that $ \env = \hull$ is true iff $\bd  U = \bd  \hull$. Since $\bd U \ci \attr$, we have $\bd \hull \ci \attr$, which is condition \eqref{itm:parsetcomp-bdC-in-F}.
  \end{proof}
\end{theorem}

 \section{Generalization of the tiling construction}
\label{sec:Generalization-of-the-tiling-construction}
While the non-triviality condition is not very restrictive, the tileset condition puts a serious constraint on the class of sets for which the canonical tiling exists. For the purpose of obtaining tube formulas for \attr, the compatibility conditions need to be satisfied; this limits the applicability of the tiling construction even further. It is natural to ask whether the tiling construction can be modified to work for more general sets. The NTC and TSC are both necessary restrictions, and each is given in terms of the convex hull \hull of \attr. While neither condition can be omitted, they \emph{can} be applied to a different initial set for the tiling construction, in place of \hull. Provided \attr satisfies OSC, it turns out that any feasible open set $O$ of \attr can be used as the initial set. In this section, we show that the tiling construction can still be carried out in this generalized setting. In the next section, we examine the analogue of the compatibility theorem for this generalization.

The main result of this section is Theorem~\ref{thm:generaltiling}, which can be paraphrased as follows: if \attr is a self-affine set with empty interior and which satisfies the OSC with feasible set $O$, then there exists a self-affine tiling of $\cj{O}$. In other words, we generalize the tiling construction by replacing the convex hull with the set $\ghull = \cj{O}$, where $O$ is an arbitrary feasible open set for \attr. The open set condition takes the role of the tileset condition and we obtain an open tiling of \ghull. 
The canonical self-affine tiling of the convex hull of \attr appears as the special case of this construction in case $\inn \hull$ is a feasible open set. We will need the following well known fact.

\begin{prop}[cf.\ Hutchinson {\cite[5.1(3)(ii)]{Hut}}]
  \label{thm:F-in-clo(O)}
  If \attr satisfies OSC with feasible open set $O$, then $\attr \ci \cj O$.
\end{prop}


Let $\simtset$ be a self-affine function system satisfying OSC and \attr be its attractor. Let $O$ be any feasible open set for \attr, i.e.\ $O$ satisfies
\eqref{eqn:def:OSC-containment} and \eqref{eqn:def:OSC-disjoint}. Set $\ghull:=\cj{O}$ from now on. Since $O \ci \inn(\ghull)$, it is clear that \ghull is the closure of its interior, $\ghull=\cj{ \inn(\ghull)}$, and that $\attr \ci \ghull$, by Proposition~\ref{thm:F-in-clo(O)}.
It is easily seen that \eqref{eqn:def:OSC-disjoint} implies
\begin{equation}\label{eqn:osc3}
  \simt_w(O)\cap\simt_v(O) =\emptyset \mbox{ for all } w,v \in W^j,v\neq w, j \in\N.
\end{equation}
Write $O^k:=\simt^k(O)$ and $\ghull^k:=\simt^k(\ghull)$ for $k=0,1,2,\ldots$.

\begin{prop}[Nestedness]\label{thm:nestedness-of-Ks}
  $\ghull^{k+1} \ci \ghull^k \ci \ghull.$
  \begin{proof}
    Note that \eqref{eqn:def:OSC-containment} implies $O^{k+1} \ci O^k$.
  \end{proof}
\end{prop}

Proposition~\ref{thm:nestedness-of-Ks} extends \cite[Thm.~5.1]{SST} and shows that $\ghull\supseteq \ghull^1\supseteq \ghull^2 \supseteq \ldots$ is a decreasing sequence of sets which converges to \attr; note that $\attr=\bigcap_{k=0}^ \infty \ghull^k$, by the contraction principle. In analogy with the tiling construction for the convex hull, the following non-triviality condition is required for a tiling of $O$ to exist:

\begin{defn}\label{def:nontriv}
  A self-affine set \attr satisfying OSC is said to be \emph{non-trivial}, if there exists a feasible open set $O$ for \attr such that
  \begin{equation}\label{eqn:nontriv}
    O\not \ci \simt(\cj{O})
  \end{equation}
  \attr is called \emph{trivial} otherwise.
\end{defn}

In fact, non-triviality implies that \eqref{eqn:nontriv} holds for all feasible sets $O$ of \attr. This is a consequence of the following proposition which characterizes the trivial case: a set \attr is trivial iff it has interior points. Hence triviality and non-triviality are independent of the particular choice of the set $O$.

\begin{prop}[Characterization of Triviality]
  \label{thm:nontriv-equiv-conditions}
  Let \attr be a self-affine set satisfying OSC. Then the following assertions are equivalent:
  \begin{enumerate}[(i)]
  \item \attr is trivial.
  \item $ \inn \attr \neq \emptyset$.
  \item $\cj{O}=\attr$ for some feasible open set $O$ of \attr.
  \item $\cj{O}=\attr$ for each feasible open set $O$ of \attr.
  \end{enumerate}
\end{prop}
\begin{proof}
  $(i)\Rightarrow (iv)$: Let $O$ be an arbitrary feasible set for \attr; we already have one containment from Proposition~\ref{thm:F-in-clo(O)}. Assume \attr is trivial, which means $O \ci \simt(\cj{O})$. Taking the closure, we get $\ghull \ci \simt(\ghull)$. The OSC implies $\simt(O) \ci O$ which, by taking closures again, implies $\simt(\ghull) \ci \ghull$. Hence $\ghull=\simt(\ghull)$. By the uniqueness of the invariant set, we infer $\attr = \ghull$.

  $(iv)\Rightarrow (iii)$ and $(iii)\Rightarrow (ii)$ are trivial. For the latter note that a feasible set $O$ is nonempty.

  $(ii)\Rightarrow (i)$ (by contraposition):
  If \attr is non-trivial, then there is a feasible set $O$ such that $O\not \ci \simt(\ghull)$.  Hence the set $T_0:=O\setminus\simt(\ghull)$ is nonempty, but $T_0 \cap \attr=\es$, since $\attr \ci \simt(\ghull)$. Observe that OSC implies $\simt_i(O) \cap \simt_j(\attr)=\emptyset$ for $i\neq j$. Therefore, $\simt_j(T_0)\cap \attr \ci \simt_j(T_0)\cap \simt_j(\attr)=\simt_j(T_0\cap \attr)=\emptyset$ and so $\simt(T_0)\cap \attr=\emptyset$. By induction, we get $\simt^k(T_0) \cap \attr = \es$ for $k=0,1,2,\ldots$ 
  Now let $x \in \attr$. Since, by the contraction principle,  $d_H(\attr,\simt^k(\cj{T_0})) = d_H(\simt^k(\attr),\simt^k(\cj{T_0})) \to 0$ as $k\to \infty$,
  there exists a sequence $x_k\to x$ with $x_k \in \simt^k(\cj{T_0})=\cj{\simt^k(T_0)}$. For each $x_k$ there are points in $\simt^k(T_0)$ arbitrarily close to $x_k$. Hence $x$ is not an interior point of \attr.
\end{proof}

\begin{remark}
  Note that Proposition~\ref{thm:nontriv-equiv-conditions} provides an easy criterion to decide whether a self-affine set has interior points. Take an arbitrary feasible open set $O$ of \attr and check whether $O$ contains a point with positive distance to \attr. If not, then \attr has interior points, otherwise $ \inn\attr$ is empty. Conversely, if it is known for some \attr that it has nonempty interior, then the search for a feasible open set can be restricted to subsets of \attr.
\end{remark}

For completeness, we note that also Corollary~\ref{thm:dimension-d-implies-trivial} generalizes to this more general notion of non-triviality used here. The argument in the proof carries over, when taking Proposition~\ref{thm:nontriv-equiv-conditions} into account.
\begin{cor}\label{cor:OSC-dimension-d-implies-trivial}
  Let $\attr \ci \bRd$ be a self-affine set satisfying OSC. If \attr has Hausdorff dimension strictly less than $d$, then \attr is non-trivial.
 Moreover, if \attr is self-similar, then also the converse holds.
\end{cor}

Now we can state the main result of this section.

\begin{theorem}[Generalized Tiling]
  \label{thm:generaltiling}
  Let \attr be a self-affine set satisfying $ \inn\attr=\emptyset$ and OSC. Let $O$ be an arbitrary feasible open set for \attr and $\ghull=\cj{O}$. Let $G_1, G_2, \ldots$ denote the connected components of the open set $O \setminus \simt(\ghull)$. Then
  \[\tiling(O):=\{\simt_w(G_q) \suth w \in W, q \in Q\}\]
  is an open tiling of \ghull, i.e., the tiles $\simt_w(G_q)$ are pairwise disjoint and
  \[\ghull=\cj{\bigcup_{R \in\tiling(O)} R}.\]
\end{theorem}

To prepare the proof, we note the following fact.

\begin{lemma} \label{lem:disjoint}
  Let $A=A^0\supseteq A^1\supseteq A^2\supseteq\ldots$ be a decreasing sequence of sets, and define $B:=\bigcap_{k=0}^ \infty A^k$. Then we can decompose $A$ as the disjoint union
  \linenopar
  \begin{align*}
    A = B \cup \bigcup_{k=0}^ \infty \left(A^k\setminus A^{k+1}\right).
  \end{align*}
\end{lemma}


Denote $T_0 := \bigcup_{q \in Q} G_q = \inn(O \setminus \simt(O))$, and more generally, set $T_k := O^k \setminus \ghull^{k+1} = \inn(\ghull^k \setminus \ghull^{k+1})$ for $k=1,2,\dots$. We now adapt \cite[Thm.~5.14, p.~3165]{SST} to the present more general setting.

\begin{lemma}[Propagation of Tilesets]
  \label{lemma:propagation_of_tilesets}
  $\simt(\tileset_k) = \tileset_{k+1}$, for each $k=0, 1, 2,\dots$.
\end{lemma}
\begin{proof}
    ($\ci$) Let $x \in \simt(\tileset_k)$. Choose $j$ so that $x \in \simt_j(\tileset_k) = \simt_j(O^k) \less \simt_j(\ghull^{k+1})$.
    Then $x \in \simt_j(O^k) \ci O^{k+1}$.
    To see $x \notin \ghull^{k+2}$, suppose it is. Then $x \in \simt_\ell(\ghull^{k+1})$ for some $\ell$.
    Note that $\ell \neq j$, since, by the choice of $j$, $x \notin \simt_j(\ghull^{k+1})$.
    Now $\ghull^{k+1} \ci \ghull^k$ by Proposition~\ref{thm:F-in-clo(O)}, which implies $x \in \simt_\ell(\ghull^{k+1}) \ci \simt_\ell(\ghull^k)$ and hence $x \in \simt_\ell(\ghull^k)\cap \simt_j(O^k)$. Since $\simt_j(O^k)$ is open and $\simt_\ell(\ghull^k)$ the closure of its interior, there must be points of $\simt_j(O^k)$ in the interior of $\simt_\ell(\ghull^k)$, i.e., in $\simt_\ell(O^k)$, contradicting OSC.

    ($\ce$) Pick $x \in \tileset_{k+1} = O^{k+1} \less \ghull^{k+2}$.
    Then $x \in O^{k+1} = \simt(O^{k})$ and so $x \in \simt_j(O^k)$ for some $j$.
    Hence $x=\simt_j(y)$ for $y \in O^k$.
    If $y \in \ghull^{k+1}$, then $x=\simt_j(y) \in \ghull^{k+2}$, a contradiction to $x \in \tileset_{k+1}$.
    So $y \notin \ghull^{k+1}$, and hence $y \in O^k \less \ghull^{k+1}$.
    We conclude $x = \simt_j(y) \in \simt_j(O^k \less \ghull^{k+1}) \ci \simt(O^k \less \ghull^{k+1})$, which completes the proof.
\end{proof}

\begin{proof}[Proof of Theorem~\ref{thm:generaltiling}] Since \attr has no interior points, by Proposition~\ref{thm:nontriv-equiv-conditions}, \attr is non-trivial, i.e., the set $T_0$ is nonempty.
  Since $T_0=\bigcup_{q \in Q} G_q$, Lemma~\ref{lemma:propagation_of_tilesets} immediately implies
  \begin{equation}\label{eqn:T_k}
  T_k = \bigcup_{w \in W^k,q \in Q} \simt_w(G_q).
  \end{equation}
  In the following, $T:=\bigcup_{R \in \tiling(O)} R$ denotes the union of all the tiles in $\tiling(O)$. Since $\simt_w(G_q) \ci O\subset \ghull$, we have $T \ci \ghull$, and since \ghull is closed the inclusion $\cj{T} \ci \ghull$ is obvious.
  It remains to show the reversed inclusion. Let $x \in \ghull$. By Lemma~\ref{lem:disjoint}, either $x \in \bigcap_k \ghull^k= \attr$ or there is some $k \in \N_0$ such that $x \in \ghull^k\setminus \ghull^{k+1} \ci \cj{{T}_k}$.
  If $x \in \cj{{T}_k}$ then equation \eqref{eqn:T_k} implies $x \in\cj{\bigcup \simt_w(G_q)}$, where the union is taken over all $w \in W^k$ and $q \in Q$, and therefore $x \in \cj{T}$. If $x \in \attr$, then there exists a sequence $(x_i) \in \ghull\setminus \attr$ converging to $x$ (since \attr has no interior points). The previous argument shows $x_i \in\cj{T}$, and hence the same holds for $x=\lim x_i$. This shows $\ghull \ci \cj{T}$ and hence $\ghull=\cj{T}$.

  By Lemma~\ref{lem:disjoint}, the sets $\ghull^k\setminus \ghull^{k+1}$ are pairwise disjoint, and hence so are the sets $T_k$.
  Moreover, the union in \eqref{eqn:T_k} is disjoint, which follows immediately from \eqref{eqn:osc3} and the fact that the sets $G_q$ are pairwise disjoint and subsets of $O$. Hence the sets $\simt_w(G_q) \in\tiling(O)$ are pairwise disjoint.
\end{proof}

As a corollary to the proof we note the following for later use
\begin{cor} \label{disjointness-of-F-and-R}
  $R\cap \attr=\emptyset$ for each $R \in\tiling(O)$.
  \begin{proof}
    Recall that $\attr \ci \ghull^k$ for each $k$. Hence $T_k=O^k\setminus \ghull^{k+1}$ has empty intersection with \attr. By \eqref{eqn:T_k}, each tile $R \in\tiling(O)$ is contained in one of the sets $T_k$.
  \end{proof}
\end{cor}

\begin{remark}[Different open sets may yield the same tiling] \label{rem:different-O-same-tiling}
  Note that two feasible open sets $O$ and $O'$ do not necessarily produce different tilings. If $\cj O=\cj{O'}$, then the tilings $\tiling(O)$ and $\tiling(O')$ coincide. In both cases one obtains an open tiling of the set $\ghull=\cj O=\cj O'$.  Therefore, for most questions it suffices to restrict considerations to feasible sets $O$ satisfying  $\inn \cj{O}=O$.
\end{remark}

\begin{remark}[The dimension of the boundary of the tiling in the general case]
  \label{rem:general-tiling-boundary-dimension}
  Proposition~\ref{prop:dim-bd-T} states that $\dim_H \bd T=\max\{\dim_H \attr, d-1\}$ for the $\tiling(\hull)$ of the convex hull, but this does not extend to the generalized tilings $\tiling(O)$. In general, $\bd O$ will not be of dimension $d-1$ and the tiles of $\tiling(O)$ may still have a fractal boundary. So the statement is slightly different. For the boundary $\bd T$ of a tiling $\tiling(O)$ one has
  \[\dim_H \bd T = \max\{\dim_H \attr, \dim_H \bd T_0\}, \text{ where } \dim_H \bd T_0 \ge d-1.\]
\end{remark}

 \section{Generalizing the compatibility theorem}
\label{sec:Generalizing-the-compatibility-theorem}
In the previous section we constructed a tiling for each feasible open set of a self-affine set \attr, provided \attr is nontrivial in the sense of Definition~\ref{def:nontriv}. For each feasible open set $O$ of \attr, we denote the corresponding tiling by $\tiling(O)$. The motivation was to find a tiling which can be used to decompose the parallel sets of \attr. The theme of this section is the search for feasible open sets that are suitable for this purpose. 
We revisit the Compatibility Theorem of \S\ref{sec:compatibility} and find conditions on a feasible open set that allow for an analogue of Theorem~\ref{thm:parallel-set-compatibility}.

We start by discussing an appropriate generalization of Proposition~\ref{parsetprop}. Throughout we use the notation of the previous section. In particular, for  a self-affine set \attr and a feasible open set $O$, $\tiling(O)$ is the associated self-affine tiling, $G_q$ are the generators, $T=\bigcup_{R \in \tiling(O)} R$ is the union of the tiles and $T_{-\eps}$ the inner parallel set of $T$.
\begin{prop} \label{generalized-parsetprop}
Let \attr be the self-affine set associated to the system $\simtset$. Assume that \attr has empty interior and satisfies the OSC with a feasible set $O$. Let $\tiling(O)$ be the associated tiling and $\ghull=\cj{O}$.
   Then
  \begin{enumerate}
    \item $\attr\ci \bd T$.
    \item $\attr_\eps \cap \ghull\ci T_{-\eps}$ for $\eps\ge 0$.
    \item $\attr_\eps \cap \ghull^\complement\ci \ghull_\eps$ for $\eps \ge 0$.
  \end{enumerate}
\end{prop}
\begin{proof}
  (i) On the one hand, $\tiling(O)$ is an open tiling of \ghull and thus $\attr \ci \ghull= \overline{T}$. On the other hand, by Corollary~\ref{disjointness-of-F-and-R}, $R \cap \attr=\emptyset$ for all $R \in \tiling(O)$, i.e., $\attr\cap T=\emptyset$. Thus $\attr\ci \overline{T}\setminus T=\bd T$.

  (ii) Fix $\eps \ge 0$. Let $x \in \attr_\eps \cap \ghull$. Then, since $x \in \ghull=\overline{T}$, either $x \in \bd T$ or $x \in T$. In the former case $x \in T_{-\eps}$ is obvious, since $\bd T \ci T_{-\eps}$. In the latter case there exists a point $y \in \attr$ with $d(x,y)\le\eps$. By (i), $y$ is in $\bd T$ and so $d(x,\bd T) \le \eps$, whence $x \in T_{-\eps}$.

  (iii) is an immediate consequence of the inclusion $\attr\ci \ghull$.
\end{proof}

\begin{theorem}[Generalized Compatibility Theorem]
  \label{thm:Generalized-parallel-set-compatibility}
  Let \attr be the self-affine set associated to the system $\simtset$. Assume that \attr has empty interior and satisfies the OSC with a feasible set $O$. Let $\tiling(O)$ be the associated tiling of $O$. Then the following assertions are equivalent:
  \begin{enumerate}
    \item \label{itm:g-parsetcomp-bdT-is-F} $\bd T = \attr$.
    \item \label{itm:g-parsetcomp-bdC-in-F} $\bd \ghull \ci \attr$.
    \item \label{itm:g-parsetcomp-bdPhiCcomp-in-F} $\bd (\ghull\setminus \simt(\ghull)) \ci \attr$.
    \item \label{itm:g-parsetcomp-bdG-in-F} $\bd G_q \ci \attr$ for all $q \in Q$.
    \item \label{itm:g-parsetcomp-FcapC-is-innernbd} $\attr_\eps \cap \ghull=T_{-\eps}$ for all $\eps\ge 0$.
    \item \label{itm:g-parsetcomp-FcapCc-is-outernbd} $\attr_\eps \cap \ghull^\complement= \ghull_\eps\cap \ghull^\complement$ for all 
    $\eps\ge 0$.
  \end{enumerate}
  \begin{proof}
   Observe that in the proof of Theorem~\ref{thm:parallel-set-compatibility} the convexity of the set \hull is not used (just the inclusion $\attr \ci \hull$, which is also satisfied here by Proposition~\ref{thm:F-in-clo(O)}: $\attr \ci \cj{O} = \ghull$). Thus the proof can be carried over to the new situation by replacing \hull with \ghull and applying Proposition~\ref{generalized-parsetprop} instead of  Proposition~\ref{parsetprop} where necessary.
  \end{proof}
  \end{theorem}

Theorem~\ref{thm:Generalized-parallel-set-compatibility} does not indicate whether or not one can always find a set $O$ such that the associated tiling can be used to decompose $\attr_\eps$. That is, one might ask if for any self-affine set \attr (satisfying OSC and $ \inn \attr=\emptyset$) there is always a feasible set $O$ such that the equivalent conditions (i)--(vi) are satisfied.
Unfortunately, this is not the case in general. There are sets for which no such $O$ exists, for instance the \emph{Koch curve}. In fact, for fractals with connected complement, Proposition~\ref{prop:no-compatible-open-set} shows that the Compatibility Theorem is never satisfied. The class of sets characterized by the connectedness of $\attr^\complement$ includes all simple fractal curves (curves with no self-intersection) like the \emph{Koch curve}, all tree-like sets (dendrites) and all totally disconnected sets in $\R^d$ with $d \ge 2$. For $d \geq 3$, it even includes topologically nontrivial sets like the \emph{Menger sponge}.

\begin{prop}\label{prop:no-compatible-open-set}
  Let \attr be a self-affine set satisfying OSC and $\inn \attr=\emptyset$.
  If the complement of \attr is connected, then there is no feasible open set $O$ such that $\bd \ghull=\attr$.
  \begin{proof}
    It suffices to consider feasible open sets $O$ satisfying $ \inn \cj O=O$ (cf.~Remark~\ref{rem:different-O-same-tiling}), which implies $\bd \ghull=\bd O$.
    Since \attr has no interior points, we have $O\cap \attr^\complement\neq \emptyset$. Let $x \in O \cap \attr^\complement$. Since $O$ is not the whole set $\attr^\complement$, there is also a point $y \in \attr^\complement\setminus O$. Since $\attr^\complement$ is connected by assumption, it is also path connected. Hence there is a path from $x$ to $y$ in $\attr^\complement$ and it must cross the boundary $\bd O=\bd \ghull$ somewhere. Hence $\bd \ghull$ is not completely contained in \attr.
  \end{proof}
\end{prop}

From the proof of Proposition~\ref{prop:no-compatible-open-set}, it is clear that any feasible set $O$ satisfying compatibility must be a subset of the envelope \env of \attr, the complement of the unbounded component of $\attr^\complement$ (cf.~Def.~\ref{def:envelope}). For many self-affine sets \attr the (interior of the) envelope itself is compatible; note that the envelope \env always satisfies the compatibility condition $\bd E\subset \attr$, by definition. If \env is feasible, then there exists a tiling that can be used to describe the parallel sets of \attr.
\begin{cor}\label{thm:envelope-extension}
Let \attr be a self-affine set with $ \inn \attr=\emptyset$ and satisfying OSC.
  If $ \inn \env$ is a feasible open set for \attr, then the self-affine tiling $\tiling=\tiling( \inn \env)$ allows a decomposition of $\attr_\eps$,
  \begin{equation}\label{eqn:envelope-extension}
    \attr_\eps = T_{-\eps} \cup \env_\eps,
  \end{equation}
  which is disjoint but for the null set $\bd \env$.
\end{cor}

However, one should not be overoptimistic; the set $\attr\subset\R$ in Example~\ref{ex1} shows that $\inn E$ is not always a feasible open set (in this case, the envelope $E$ coincides with the convex hull). We also provide the following example.

\begin{exm} \label{ex:envelope_counterexample}
  Let \attr be the attractor of the system $\{\simt_1,\ldots,\simt_4\}$ of four similarities, where $\simt_1, \simt_2, \simt_3$ are the usual mappings used for the Sierpinski gasket and $\simt_4$ scales the initial triangle by a factor $\frac 14$, rotates it by $\pi$ and translates it by $(\frac 14,\frac{\sqrt{3}}{8})$ such that it fits in the largest hole in $\simt_1 \attr$ (cf.~Figure~\ref{fig:envelope_counterexample}). This set satisfies OSC, for instance, the set $O:=\bigcup_{j=1}^4 \simt_j \hull$ is feasible. The envelope of \attr coincides with the convex hull of \attr, $E=\hull$, but $E$ is not feasible, since $\simt_1 E\cap \simt_4 E$ is not empty.
\end{exm}

\begin{figure}
  \centering
  \scalebox{1.0}{ \includegraphics{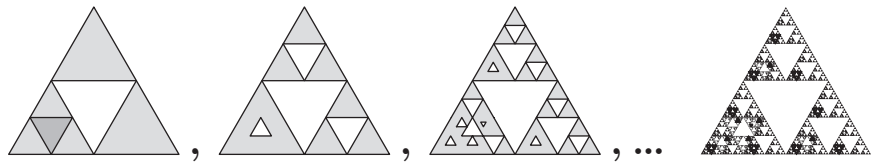}}
  \caption{\captionsize A self-similar set which satisfies OSC but for which the envelope is not feasible; see
  Example~\ref{ex:envelope_counterexample}.}
  \label{fig:envelope_counterexample}
\end{figure}

We believe that, if $\attr^\complement$ is not connected, i.e., if the envelope $E$ of \attr has nonempty interior, then there  exists always a subset $O$ of $E$ which is both compatible and feasible. So far we have not been able to prove this.


\section{Concluding comments and remarks}
\label{sec:end}

\begin{remark}
  Corollary~\ref{cor:OSC-dimension-d-implies-trivial} says that the trivial self-similar sets in $\R^d$ are precisely those which have full Hausdorff dimension. Hence all self-similar sets for which self-similar tilings can be constructed have Hausdorff dimension (and thus Minkowski dimension) strictly less than $d$. In \cite{Pointwise}, tube formulas are obtained for a class of \emph{fractal sprays} in $\R^d$, provided these sprays satisfy the same condition on the Minkowski dimension of their boundary. So Corollary~\ref{cor:OSC-dimension-d-implies-trivial} ensures that the latter condition does not impose any restrictions on the applicability of the tube formula results to the self-similar case.

  In the self-affine case, however,
  it remains open whether there exists a non-trivial $\attr\subset\R^d$ (i.e., one with empty interior) satisfying OSC which has full Hausdorff dimension.
  On the other hand, tube formulas are not available yet in this more general setting. The results obtained for fractal sprays do not apply in this case.
\end{remark}

\begin{remark}[Relation to tilings of \bRd]
  There is another notion of self-similar (self-affine) tilings which as been studied at length, namely tilings of the plane or, more generally, of \bRd. In this approach copies of self-similar (or self-affine) sets $\attr \ci \bRd$ are used as tiles to tile the whole of \bRd, which is very different to our approach, where a feasible set of \attr is tiled and where the tiles are not copies of \attr but subsets of  $\attr^\complement$. See \cite{GAST1,LaWa,NgaiTang05,Bandt97}, for example.

  However, there are interesting relations between both concepts. Firstly, the open set condition  is a natural requirement in both approaches. Secondly, the concepts are in a way complementary to each other. Tilings of \bRd require a self-similar sets \attr to have full dimension, while the tilings of feasible sets require \attr to have dimension strictly less than $d$.

  \begin{prop}\label{thm:tiling-dichotomy}
    Let \attr be a self-similar set satisfying OSC. Then there is a dichotomy:
    \begin{enumerate}
      \item $\inn \attr = \es$, in which case there is a self-similar tiling of any feasible open set $O$ of \attr, or
      \item $\inn \attr \neq \es$, in which case \attr gives a self-similar tiling of \bRd.
    \end{enumerate}
    \begin{proof}
      (i) is a corollary of Theorem~\ref{thm:generaltiling};
      (ii) is \cite[Thm.~9.1]{Bandt97} with unit tile \attr.
    \end{proof}
  \end{prop}
  The dichotomy of Proposition~\ref{thm:tiling-dichotomy} extends to the self-affine case, see \cite[Thm.~1.2]{LaWa} and \cite[Lemma~2.3]{NgaiTang05} for the case (ii). The latter result is formulated for subsets of $\R^2$ (and more general contractions) but the same argument works in $\R^d$.
  %
\end{remark}

 \bibliographystyle{plain}


\end{document}